\documentclass[11pt]{article}

\usepackage{amsmath,amsthm}
\usepackage{amssymb}
\usepackage{hyperref}

\allowdisplaybreaks

\begin{document}

\title{Integrals of polylogarithms and infinite series involving generalized harmonic numbers}
\author{
Rusen Li
\\
\small School of Mathematics\\
\small Shandong University\\
\small Jinan 250100 China\\
\small \texttt{limanjiashe@163.com}
}

\date{
\small 2020 MR Subject Classifications: 33E20, 11B83
}

\maketitle

\def\stf#1#2{\left[#1\atop#2\right]}
\def\sts#1#2{\left\{#1\atop#2\right\}}
\def\e{\mathfrak e}
\def\f{\mathfrak f}

\newtheorem{theorem}{Theorem}
\newtheorem{Prop}{Proposition}
\newtheorem{Cor}{Corollary}
\newtheorem{Lem}{Lemma}
\newtheorem{Example}{Example}
\newtheorem{Remark}{Remark}
\newtheorem{Definition}{Definition}
\newtheorem{Conjecture}{Conjecture}
\newtheorem{Problem}{Problem}

\begin{abstract}
In this paper, we give explicit evaluation for some infinite series involving generalized (alternating) harmonic numbers. In addition, some formulas for generalized (alternating) harmonic numbers will also be derived.
\\
{\bf Keywords:} polylogarithm function, generalized harmonic numbers
\end{abstract}

\section{Introduction and preliminaries}
Let $\mathbb Z$, $\mathbb N$, $\mathbb N_{0}$ and $\mathbb C$ denote the set of integers, positive integers, nonnegative integers and complex numbers, respectively. The well-known polylogarithm function is defined as
$$
Li_{p}(x):=\sum_{n=1}^\infty \frac{x^n}{n^p}\quad (\lvert x \lvert \leq 1,\quad p \in \mathbb N_{0})\,.
$$
Note that when $p=1$, $-Li_{1}(x)$ is the logarithm function $\log(1-x)$. Here, and throughout this paper, we use the natural logarithm (to base $e$). Furthermore, $Li_{n}(1)=\zeta(n)$, where $\zeta(s):=\sum_{n=1}^\infty n^{-s}$ denotes the well-known Riemann zeta function. The classical generalized harmonic numbers of order $m$ is defined by the partial sum of the Riemann Zeta function $\zeta(m)$ as:
$$
H_n^{(m)}:=\sum_{j=1}^n \frac{1}{j^m} \quad (n, m \in \mathbb N)\,.
$$
For convenience, we recall the classical generalized alternating harmonic numbers $\overline{H}_n^{(p)}=\sum_{j=1}^n (-1)^{j-1}/{j^p}$.

It is interesting that infinite series containing harmonic numbers $H_n$ can be expressed explicitly in terms of the logarithms and polylogarithm functions. For instance, De Doelder \cite{Doelder} used the integrals
\begin{align*}
&\int_{0}^{x}\frac{\log^{2}(1-t)}{t}\mathrm{d}t, \int_{0}^{x}\frac{\log(t)\log^{2}(1-t)}{t}\mathrm{d}t, \int_{0}^{1}\frac{\log^{2}(1+t)\log(1-t)}{t^{2}}\mathrm{d}t
\end{align*}
and
\begin{align*}
\int_{0}^{1}\frac{\log^{2}(1+t)\log(1-t)}{t}\mathrm{d}t
\end{align*}
to evaluate infinite series containing harmonic numbers of types $\sum_{n=1}^\infty \frac{H_{n-1}}{n^{2}}x^{n}$, $\sum_{n=1}^\infty \frac{H_{n-1}}{n^{3}}x^{n}$, $\sum_{n=1}^\infty \frac{(H_{n-1})^{2}}{n}x^{n}$ and $\sum_{n=1}^\infty \frac{(H_{n-1})^{2}}{n^{2}}x^{n}$.

When $p=1$, De Doelder \cite{Doelder} gave the following formulas:
\begin{align*}
&\sum_{n=1}^\infty \frac{H_{n-1}}{n}x^n=\frac{1}{2}\log^{2}(1-x)\quad (\lvert x \lvert \leq 1)\,,\\
&\sum_{n=1}^\infty \frac{H_{n-1}}{n^{2}}x^n=\frac{1}{2}\log(x)\log^{2}(1-x)+\log(1-x)Li_{2}(1-x)
-Li_{3}(1-x)\\
&\qquad \qquad \qquad \quad +Li_{3}(1)\quad (0 \leq x \leq 1)\,,\\
&\sum_{n=1}^\infty \frac{(-1)^{n}H_{n-1}}{n^{2}}x^n=\frac{1}{2}\log(x)\log^{2}(1+x)
-\frac{1}{3}\log^{3}(1+x)-Li_{3}\bigg(\frac{1}{1+x}\bigg)\\
&\qquad \qquad \qquad \qquad \quad -\log(1+x)Li_{2}\bigg(\frac{1}{1+x}\bigg)+Li_{3}(1)\quad (0 \leq x \leq 1)\,,\\
&\sum_{n=1}^\infty \frac{H_{n-1}}{n^{3}}=\frac{1}{360}\pi^{4}\,,\\
&\sum_{n=1}^\infty \frac{H_{n-1}(-1)^{n-1}}{n^{3}}=\frac{1}{48}\pi^{4}-2 Li_{4}\bigg(\frac{1}{2}\bigg)-\frac{7}{4}\log(2)\zeta(3)+\frac{1}{12}\pi^{2}\log(2)
-\frac{1}{12}\log^{4}(2)\,.
\end{align*}

In this paper, we give explicit evaluation for infinite series involving generalized (alternating) harmonic numbers of types
$\sum_{n=1}^\infty \frac{H_{n}}{n^{3}}x^{n}$,
$\sum_{n=1}^\infty \frac{H_{n}}{n^{3}}(-x)^{n}$,
$\sum_{n=1}^\infty \frac{H_{n}^{(2)}}{n}x^{n}$,
$\sum_{n=1}^\infty \frac{H_{n}^{(2)}}{n}(-x)^{n}$,
$\sum_{n=1}^\infty \frac{H_{n}^{(2)}}{n^{2}}x^{n}$,
$\sum_{n=1}^\infty \frac{H_{n}^{(2)}}{n^{2}}(-x)^{n}$,
$\sum_{n=1}^\infty \frac{\overline{H}_{n}}{n}x^{n}$,
$\sum_{n=1}^\infty \frac{\overline{H}_{n}^{(2)}}{n}x^{n}$,
$\sum_{n=1}^\infty \frac{\overline{H}_{n}^{(2)}}{n}(-x)^{n}$.
In addition, some formulas for generalized (alternating) harmonic numbers will also be derived.

\section{Infinite series containing generalized harmonic numbers}

Now we establish more explicit formulas for infinite series $\sum_{n=1}^\infty \frac{H_{n}^{(p)}}{n^{m}}x^{n}$.

\begin{theorem}\label{thm1}
Let $p \in \mathbb N$ with $p$ odd and $\lvert x \lvert \leq 1$, then we have
\begin{align*}
\sum_{n=1}^\infty \frac{H_{n}^{(p)}}{n}x^{n}
=\frac{1}{2}\sum_{j=1}^{p}(-1)^{j-1}Li_{j}(x)Li_{p+1-j}(x)+Li_{p+1}(x)\,.
\end{align*}
\end{theorem}
\begin{proof}
Integrating the generating function of $H_{n}^{(p)}$, we can write
\begin{align*}
&\quad \sum_{n=1}^\infty \frac{H_{n}^{(p)}}{n+1}x^{n+1}=\int_{0}^{x}\frac{Li_{p}(t)}{1-t}\mathrm{d}t\\
&=Li_{1}(x)Li_{p}(x)-\int_{0}^{x}\frac{Li_{1}(t)Li_{p-1}(t)}{t}\mathrm{d}t\\
&=Li_{1}(x)Li_{p}(x)-Li_{2}(x)Li_{p-1}(x)+\int_{0}^{x}\frac{Li_{2}(t)Li_{p-2}(t)}{t}\mathrm{d}t\\
&=\sum_{j=1}^{k}(-1)^{j-1}Li_{j}(x)Li_{p+1-j}(x)+(-1)^{k}\int_{0}^{x}\frac{Li_{k}(t)Li_{p-k}(t)}{t}\mathrm{d}t \quad (0 \leq k \leq p)\\
&=\sum_{j=1}^{p}(-1)^{j-1}Li_{j}(x)Li_{p+1-j}(x)
+(-1)^{p}\int_{0}^{x}\frac{Li_{p}(t)}{1-t}\mathrm{d}t\,.
\end{align*}
Note that
$$
\sum_{n=1}^\infty \frac{H_{n}^{(p)}}{n}x^{n}=\sum_{n=1}^\infty \frac{H_{n-1}^{(p)}}{n}x^{n}+\sum_{n=1}^\infty \frac{x^{n}}{n^{p+1}}\,,
$$
thus we get the desired result.
\end{proof}

\begin{Cor}
Let $p, n \in \mathbb N$ and $k \in \mathbb N_{0}$ with $0 \leq k \leq p$, then we have
\begin{align*}
H_{n}^{(p)}
=\sum_{j=1}^{k}(-1)^{j-1}\sum_{\ell=1}^{n}\frac{n+1}{\ell^{j}(n+1-\ell)^{p+1-j}}
+(-1)^{k}\sum_{\ell=1}^{n}\frac{1}{\ell^{k}(n+1-\ell)^{p-k}}\,.
\end{align*}
In particular, if $m \in \mathbb N_{0}$, then we can obtain that
\begin{align*}
&\quad H_{n}^{(2m+1)}\\
&=\sum_{j=1}^{m}(-1)^{j-1}\sum_{k=1}^{n-1}\frac{n}{k^{j}(n-k)^{2m+2-j}}
+\frac{(-1)^{m}}{2}\sum_{k=1}^{n-1}\frac{n}{k^{m+1}(n-k)^{m+1}}+\frac{1}{n^{2m+1}}\,,\\
&H_{n}^{(2m)}
=\sum_{j=1}^{m}(-1)^{j-1}\sum_{k=1}^{n-1}\frac{n}{k^{j}(n-k)^{2m+1-j}}
+(-1)^{m}\sum_{j=1}^{n-1}\frac{1}{j^{m}(n-j)^{m}}+\frac{1}{n^{2m}}\,.
\end{align*}
\end{Cor}
\begin{proof}
The following formula is obtained in the proof of Theorem \ref{thm1},
\begin{align*}
&\quad \sum_{n=1}^\infty \frac{H_{n}^{(p)}}{n+1}x^{n+1}\\
&=\sum_{j=1}^{k}(-1)^{j-1}Li_{j}(x)Li_{p+1-j}(x)+(-1)^{k}\int_{0}^{x}\frac{Li_{k}(t)Li_{p-k}(t)}{t}\mathrm{d}t \quad (0 \leq k \leq p)\,,
\end{align*}
comparing the coefficients on both sides gives the desired result.
\end{proof}

Before going further, we provide some lemmas.
\begin{Lem}(\cite[p.204]{Lewin})\label{lem0}
Let $0 \leq x < 1$, then we have
\begin{align*}
&\quad \int_{0}^{x}\frac{\log^{2}(t)\log(1-t)}{1-t}\mathrm{d}t\\
&=-2\bigg(Li_{4}(x)+Li_{4}\bigg(\frac{-x}{1-x}\bigg)-Li_{4}(1-x)+Li_{4}(1)
-Li_{3}(x)\log(1-x)\bigg)\\
&\quad -2 Li_{3}(1-x)\log(x)+2Li_{2}(1-x)\log(x)\log(1-x)-\frac{1}{6}\pi^{2}\log^{2}(1-x)\\
&\quad +\frac{1}{2}\log^{2}(x)\log^{2}(1-x)+\frac{1}{3}\log(x)\log^{3}(1-x)-\frac{1}{12}\log^{4}(1-x)\\
&\quad +2Li_{3}(1)\bigg(\log(x)-\log(1-x)\bigg)\,.
\end{align*}
\end{Lem}

\begin{Lem}(\cite[p.310]{Lewin})\label{lem00}
Let $0 \leq x < 1$, then we have
\begin{align*}
&\int_{0}^{x}\frac{\log^{2}(t)\log(1-t)}{t}\mathrm{d}t
=-2 Li_{4}(x)+2 Li_{3}(x)\log(x)-Li_{2}(x)\log^{2}(x)\,.
\end{align*}
\end{Lem}

\begin{Lem}(\cite{Doelder,Lewin})\label{lem000}
Let $0 \leq x \leq 1$, then we have
\begin{align}
\int_{0}^{x}\frac{\log^{2}(1-t)}{t}\mathrm{d}t
&=\log(x)\log^{2}(1-x)+2\log(1-x)Li_{2}(1-x)\notag\\
&\quad -2Li_{3}(1-x)+2Li_{3}(1)\,,\label{Lewin1}\\
\int_{0}^{x}\frac{\log^{2}(1+t)}{t}\mathrm{d}t
&=\log(x)\log^{2}(1+x)-\frac{2}{3}\log^{3}(1+x)-2Li_{3}\bigg(\frac{1}{1+x}\bigg)\notag\\
&\quad -2\log(1+x)Li_{2}\bigg(\frac{1}{1+x}\bigg)+2Li_{3}(1)\,.\label{Lewin2}
\end{align}
\end{Lem}

De Doelder \cite{Doelder} only calculated the integral $\int_{0}^{1}\frac{\log^{2}(t)\log(1+t)}{1+t}\mathrm{d}t$, we now give explicit expression for the intagral $\int_{0}^{x}\frac{\log^{2}(t)\log(1+t)}{1+t}\mathrm{d}t$.
\begin{Lem}\label{lem0000}
Let $0 \leq x \leq 1$, then we have
\begin{align*}
&\quad \int_{0}^{x}\frac{\log^{2}(t)\log(1+t)}{1+t}\mathrm{d}t\\
&=2 \bigg(Li_{4}(-x)+Li_{4}\bigg(\frac{x}{1+x}\bigg)+Li_{4}\bigg(\frac{1}{1+x}\bigg)-Li_{4}(1)
-Li_{3}(1)\log(x)\\
&\quad +Li_{3}\bigg(\frac{x}{1+x}\bigg)\log(1+x)+Li_{3}\bigg(\frac{1}{1+x}\bigg)\log(1+x)
+Li_{3}\bigg(\frac{1}{1+x}\bigg)\log(x)\\
&\quad +Li_{2}\bigg(\frac{1}{1+x}\bigg)\log(x)\log(1+x)\bigg)+\frac{1}{6}\pi^{2}\log^{2}(1+x)
-\frac{1}{2}\log^{2}(x)\log^{2}(1+x)\\
&\quad +\frac{4}{3}\log(x)\log^{3}(1+x)-\frac{1}{2}\log^{4}(1+x)\,.
\end{align*}
\end{Lem}
\begin{proof}
Following De Doelder's paper \cite{Doelder}, we make the substitution $t=\frac{1}{u}-1$ and it yields that
\begin{align*}
&\quad \int_{0}^{x}\frac{\log^{2}(t)\log(1+t)}{1+t}\mathrm{d}t\\
&=-\int_{\frac{1}{1+x}}^{1}\frac{\big(\log(1-u)-\log(u)\big)^{2}\log(u)}{u}\mathrm{d}u\\
&=-\int_{\frac{1}{1+x}}^{1}\frac{\log^{2}(1-u)\log(u)}{u}\mathrm{d}u
+2\int_{\frac{1}{1+x}}^{1}\frac{\log(1-u)\log^{2}(u)}{u}\mathrm{d}u\\
&\quad -\int_{\frac{1}{1+x}}^{1}\frac{\log^{3}(u)}{u}\mathrm{d}u\\
&=-\int_{0}^{\frac{x}{1+x}}\frac{\log^{2}(y)\log(1-y)}{1-y}\mathrm{d}y
+2\int_{\frac{1}{1+x}}^{1}\frac{\log(1-u)\log^{2}(u)}{u}\mathrm{d}u\\
&\quad +\frac{1}{4}\log^{4}(1+x)\,.
\end{align*}
With the help of Lemmata \ref{lem0} and \ref{lem00}, we get the desired result.
\end{proof}

\begin{theorem}
Let $0 \leq x \leq 1$, then we have
\begin{align*}
&\quad \sum_{n=1}^\infty \frac{H_{n}}{n^{3}}x^{n}\\
&=2Li_{4}(x)+Li_{4}\bigg(\frac{-x}{1-x}\bigg)-Li_{4}(1-x)+Li_{4}(1)-Li_{3}(x)\log(1-x)\\
&\quad +Li_{3}(1)\log(1-x)+\frac{1}{24}\log^{4}(1-x)-\frac{1}{6}\log(x)\log^{3}(1-x)\\
&\quad +\frac{1}{12}\pi^{2}\log^{2}(1-x)\,,\\
&\quad \sum_{n=1}^\infty \frac{H_{n}}{n^{3}}(-x)^{n}\\
&=2Li_{4}(-x)+Li_{4}\bigg(\frac{1}{1+x}\bigg)+Li_{4}\bigg(\frac{x}{1+x}\bigg)-Li_{4}(1)\\
&\quad +\log(1+x)Li_{3}\bigg(\frac{1}{1+x}\bigg)+\log(1+x)Li_{3}\bigg(\frac{x}{1+x}\bigg)\\
&\quad +\frac{1}{12}\pi^{2}\log^{2}(1+x)+\frac{1}{3}\log(x)\log^{3}(1+x)-\frac{1}{4}\log^{4}(1+x)\,.
\end{align*}
\end{theorem}
\begin{proof}
Integrating the generating function of $H_{n}$, we can write
\begin{align*}
&\quad \sum_{n=1}^\infty \frac{H_{n}}{(n+1)^{3}}x^{n+1}\\
&=\frac{1}{2}\int_{0}^{x}\frac{\mathrm{d}u}{u}\int_{0}^{u}\frac{\log^{2}(1-t)}{t}\mathrm{d}t\\
&=\frac{1}{2}\log(x)\int_{0}^{x}\frac{\log^{2}(1-t)}{t}\mathrm{d}t
-\frac{1}{2}\int_{0}^{x}\frac{\log(t)\log^{2}(1-t)}{t}\mathrm{d}t\\
&=\frac{1}{2}\log(x)\int_{0}^{x}\frac{\log^{2}(1-t)}{t}\mathrm{d}t
-\frac{1}{4}\log^{2}(x)\log^{2}(1-x)-\frac{1}{2}\int_{0}^{x}\frac{\log^{2}(t)\log(1-t)}{1-t}\mathrm{d}t\,,\\
&\quad \sum_{n=1}^\infty \frac{H_{n}}{(n+1)^{3}}(-x)^{n+1}\\
&=\frac{1}{2}\int_{0}^{x}\frac{\mathrm{d}u}{u}\int_{0}^{u}\frac{\log^{2}(1+t)}{t}\mathrm{d}t\\
&=\frac{1}{2}\log(x)\int_{0}^{x}\frac{\log^{2}(1+t)}{t}\mathrm{d}t
-\frac{1}{4}\log^{2}(x)\log^{2}(1+x)+\frac{1}{2}\int_{0}^{x}\frac{\log^{2}(t)\log(1+t)}{1+t}\mathrm{d}t\,.
\end{align*}
Note that
$$
\sum_{n=1}^\infty \frac{H_{n}}{n^{3}}x^{n}
=\sum_{n=1}^\infty \frac{H_{n}}{(n+1)^{3}}x^{n+1}+\sum_{n=1}^\infty \frac{x^{n}}{n^{4}}\,,
$$
and
$$
\sum_{n=1}^\infty \frac{H_{n}}{n^{3}}(-x)^{n}
=\sum_{n=1}^\infty \frac{H_{n}}{(n+1)^{3}}(-x)^{n+1}+\sum_{n=1}^\infty \frac{(-x)^{n}}{n^{4}}\,,
$$
with the help of Lemmata \ref{lem0}, \ref{lem000} and \ref{lem0000}, we get the desired result.
\end{proof}

\begin{theorem}\label{thm2}
Let $0 \leq x \leq 1$, then we have
\begin{align*}
&\quad \sum_{n=1}^\infty \frac{H_{n}^{(2)}}{n}x^{n}\\
&=-Li_{2}(x)\log(1-x)-\log(x)\log^{2}(1-x)-2\log(1-x)Li_{2}(1-x)\\
&\quad +2Li_{3}(1-x)-2Li_{3}(1)+Li_{3}(x)\,,\\
&\quad \sum_{n=1}^\infty \frac{H_{n}^{(2)}}{n}(-x)^{n}\\
&=-Li_{2}(-x)\log(1+x)-\log(x)\log^{2}(1+x)+\frac{2}{3}\log^{3}(1+x)\\
&\quad +2\log(1+x)Li_{2}\bigg(\frac{1}{1+x}\bigg)+2Li_{3}\bigg(\frac{1}{1+x}\bigg)-2Li_{3}(1)+Li_{3}(-x)\,.
\end{align*}
\end{theorem}
\begin{proof}
Integrating the generating function of $H_{n}^{(2)}$, we can write
\begin{align*}
&\sum_{n=1}^\infty \frac{H_{n}^{(2)}}{n+1}x^{n+1}
=-\log(1-x)Li_{2}(x)-\int_{0}^{x}\frac{\log^{2}(1-t)}{t}\mathrm{d}t\,,\\
&\sum_{n=1}^\infty \frac{H_{n}^{(2)}}{n+1}(-x)^{n+1}
=-\log(1+x)Li_{2}(-x)-\int_{0}^{x}\frac{\log^{2}(1+t)}{t}\mathrm{d}t\,.
\end{align*}
With the help of Lemma \ref{lem000}, we get the desired result.
\end{proof}

\begin{theorem}
Let $0 \leq x \leq 1$, then we have
\begin{align*}
&\quad \sum_{n=1}^\infty \frac{H_{n}^{(2)}}{n^{2}}x^{n}\\
&=-Li_{4}(x)-2 Li_{4}\bigg(\frac{-x}{1-x}\bigg)+2 Li_{4}(1-x)-2 Li_{4}(1)
+2 Li_{3}(x)\log(1-x)\\
&\quad -2 Li_{3}(1)\log(1-x)+\frac{1}{2}Li_{2}(x)^{2}-\frac{1}{6}\pi^{2}\log^{2}(1-x)\\
&\quad +\frac{1}{3}\log(x)\log^{3}(1-x)-\frac{1}{12}\log^{4}(1-x)\,,\\
&\quad \sum_{n=1}^\infty \frac{H_{n}^{(2)}}{n^{2}}(-x)^{n}\\
&=-Li_{4}(-x)-2 Li_{4}\bigg(\frac{x}{1+x}\bigg)-2 Li_{4}\bigg(\frac{1}{1+x}\bigg)+2 Li_{4}(1)\\
&\quad -2 Li_{3}\bigg(\frac{x}{1+x}\bigg)\log(1+x)-2 Li_{3}\bigg(\frac{1}{1+x}\bigg)\log(1+x)+\frac{1}{2}Li_{2}(-x)^{2}\\
&\quad -\frac{1}{6}\pi^{2}\log^{2}(1+x)
-\frac{2}{3}\log(x)\log^{3}(1+x)+\frac{1}{2}\log^{4}(1+x)\,.
\end{align*}
\end{theorem}
\begin{proof}
Integrating the generating function of $H_{n}^{(2)}$, we can write
\begin{align*}
&\quad \sum_{n=1}^\infty \frac{H_{n}^{(2)}}{(n+1)^{2}}x^{n+1}\\
&=\int_{0}^{x}\frac{\mathrm{d}u}{u}\int_{0}^{u}\frac{Li_{2}(t)}{1-t}\mathrm{d}t\\
&=\log(x)\int_{0}^{x}\frac{Li_{2}(t)}{1-t}\mathrm{d}t-\int_{0}^{x}\frac{\log(t)Li_{2}(t)}{1-t}\mathrm{d}t\\
&=\log(x)\int_{0}^{x}\frac{Li_{2}(t)}{1-t}\mathrm{d}t+\log(x)\log(1-x)Li_{2}(x)+\frac{1}{2}Li_{2}(x)^{2}\\
&\quad +\int_{0}^{x}\frac{\log(t)\log^{2}(1-t)}{t}\mathrm{d}t\,,\\
&\quad \sum_{n=1}^\infty \frac{H_{n}^{(2)}}{(n+1)^{2}}(-x)^{n+1}\\
&=\log(x)\int_{0}^{-x}\frac{Li_{2}(t)}{1-t}\mathrm{d}t+\int_{0}^{x}\frac{\log(t)Li_{2}(-t)}{1+t}\mathrm{d}t\\
&=\log(x)\int_{0}^{-x}\frac{Li_{2}(t)}{1-t}\mathrm{d}t+\log(x)\log(1+x)Li_{2}(-x)
+\frac{1}{2}Li_{2}(-x)^{2}\\
&\quad +\int_{0}^{x}\frac{\log(t)\log^{2}(1+t)}{t}\mathrm{d}t\,.
\end{align*}
With the help of Lemmata \ref{lem0}, \ref{lem0000} and Theorem \ref{thm2}, we get the desired result.
\end{proof}

\begin{Remark}
It seems difficult to give explicit expressions for infinite series of types
$\sum_{n=1}^\infty \frac{H_{n}}{n^{4}}x^{n}$ and
$\sum_{n=1}^\infty \frac{H_{n}^{(2)}}{n^{3}}x^{n}$, since the integrals $\int_{0}^{x}\frac{\log^{2}(t)\log^{2}(1-t)}{t}\mathrm{d}t$ and $\int_{0}^{x}\frac{\log^{2}(t)Li_{2}(t)}{1-t}\mathrm{d}t$ are not known to be related to the polylogarithm functions.
\end{Remark}

\begin{Example}
Some illustrative examples are as following:
\begin{align*}
&\sum_{n=1}^\infty \frac{H_{n}}{n^{3}\cdot2^{n}}
=Li_{4}\bigg(\frac{1}{2}\bigg)+\frac{1}{720}\pi^{4}-\frac{1}{8}\log(2)\zeta(3)
+\frac{1}{24}\log^{4}(2)\,,\\
&\sum_{n=1}^\infty \frac{H_{n}}{n^{3}\cdot2^{n}}(-1)^{n}
=2Li_{4}\bigg(-\frac{1}{2}\bigg)+Li_{4}\bigg(\frac{1}{3}\bigg)+Li_{4}\bigg(\frac{2}{3}\bigg)
-\frac{1}{90}\pi^{4}\\
&\qquad \qquad \qquad \quad +\log\bigg(\frac{3}{2}\bigg)Li_{3}\bigg(\frac{2}{3}\bigg)+\log\bigg(\frac{3}{2}\bigg)Li_{3}\bigg(\frac{1}{3}\bigg)\\
&\qquad \qquad \qquad \quad +\frac{1}{12}\pi^{2}\log^{2}\bigg(\frac{3}{2}\bigg)
-\frac{1}{12}\log^{4}\bigg(\frac{3}{2}\bigg)-\frac{1}{6}\log(6)\log^{3}\bigg(\frac{3}{2}\bigg)\,,\\
&\sum_{n=1}^\infty \frac{H_{n}}{n^{3}}(-1)^{n}
=2Li_{4}\bigg(\frac{1}{2}\bigg)-\frac{11}{360}\pi^{4}+\frac{7}{4}\log(2)\zeta(3)\\
&\qquad \qquad \qquad \quad -\frac{1}{12}\pi^{2}\log^{2}(2)+\frac{1}{12}\log^{4}(2)\,,\\
&\sum_{n=1}^\infty \frac{H_{n}^{(2)}}{(n+1)2^{n+1}}
=-\frac{1}{4}\zeta(3)+\frac{1}{12}\pi^{2}\log(2)-\frac{1}{6}\log^{3}(2)\,,\\
&\sum_{n=1}^\infty \frac{H_{n}^{(2)}}{n\cdot2^{n}}
=\frac{5}{8}\zeta(3)\,,\\
&\sum_{n=1}^\infty \frac{H_{n}^{(2)}(\sqrt{5}-1)^{n+1}}{(n+1)2^{n+1}}
=-\frac{2}{5}\zeta(3)-\frac{1}{5}\pi^{2}\log\bigg(\frac{\sqrt{5}-1}{2}\bigg)+\frac{2}{3}\log^{3}\bigg(\frac{\sqrt{5}-1}{2}\bigg)\,,\\
&\sum_{n=1}^\infty \frac{H_{n}^{(2)}(\sqrt{5}-1)^{n}}{n\cdot2^{n}}
=-\frac{2}{5}\zeta(3)-\frac{1}{5}\pi^{2}\log\bigg(\frac{\sqrt{5}-1}{2}\bigg)+\frac{2}{3}\log^{3}\bigg(\frac{\sqrt{5}-1}{2}\bigg)\\
&\qquad \qquad \qquad\qquad\quad +Li_{3}\bigg(\frac{\sqrt{5}-1}{2}\bigg)\,,\\
&\sum_{n=1}^\infty \frac{H_{n}^{(2)}}{n}(-1)^{n}=\frac{1}{12}\pi^{2}\log(2)-\zeta(3)\,,\\
&\sum_{n=1}^\infty \frac{H_{n}^{(2)}}{n^{2}\cdot2^{n}}
=Li_{4}\bigg(\frac{1}{2}\bigg)+\frac{1}{1440}\pi^{4}+\frac{1}{4}\log(2)\zeta(3)
-\frac{1}{24}\pi^{2}\log^{2}(2)+\frac{1}{24}\log^{4}(2)\,,\\
&\sum_{n=1}^\infty \frac{H_{n}^{(2)}}{n^{2}\cdot2^{n}}(-1)^{n}=
-Li_{4}\bigg(-\frac{1}{2}\bigg)-2\bigg(Li_{4}\bigg(\frac{1}{3}\bigg)+Li_{4}\bigg(\frac{2}{3}\bigg)
-\frac{1}{90}\pi^{4}\bigg)\\
&\qquad \qquad \qquad\qquad\quad -2\bigg(\log\bigg(\frac{3}{2}\bigg)Li_{3}\bigg(\frac{2}{3}\bigg)
+\log\bigg(\frac{3}{2}\bigg)Li_{3}\bigg(\frac{1}{3}\bigg)\bigg)\\
&\qquad \qquad \qquad\qquad\quad +\frac{1}{2}Li_{2}\bigg(-\frac{1}{2}\bigg)^{2}-\frac{1}{6}\pi^{2}\log^{2}\bigg(\frac{3}{2}\bigg)
+\frac{1}{2}\log^{4}\bigg(\frac{3}{2}\bigg)\\
&\qquad \qquad \qquad\qquad\quad +\frac{2}{3}\log(2)\log^{3}\bigg(\frac{3}{2}\bigg)\,.
\end{align*}
\end{Example}

\section{Infinite series containing generalized alternating harmonic numbers}
\begin{Prop}\label{prop}
Let $p, n \in \mathbb N$ and $k \in \mathbb N_{0}$ with $0 \leq k \leq p$, then we have
\begin{align*}
&\overline{H}_{n}^{(p)}
=\sum_{j=1}^{k}(-1)^{j}\sum_{\ell=1}^{n}\frac{(-1)^{n+1-\ell}(n+1)}{\ell^{j}(n+1-\ell)^{p+1-j}}
+(-1)^{k+1}\sum_{\ell=1}^{n}\frac{(-1)^{n+1-\ell}}{\ell^{k}(n+1-\ell)^{p-k}}\,.
\end{align*}
In particular, we have
\begin{align*}
&\overline{H}_{n}^{(p)}
=\frac{n+1}{2}\sum_{j=1}^{p}(-1)^{j}\sum_{\ell=1}^{n}\frac{(-1)^{n+1-\ell}}{\ell^{j}(n+1-\ell)^{p+1-j}}\quad (p+n\quad\hbox{even})\,,\\
&\sum_{j=1}^{p}(-1)^{j}\sum_{\ell=1}^{n}\frac{(-1)^{n+1-\ell}}{\ell^{j}(n+1-\ell)^{p+1-j}}
=0\quad (p+n\quad\hbox{odd})\,.
\end{align*}
\end{Prop}
\begin{proof}
Integrating the generating function of $\overline{H}_{n}^{(p)}$, we can write
\begin{align*}
&\quad \sum_{n=1}^\infty \frac{\overline{H}_{n}^{(p)}}{n+1}x^{n+1}=\int_{0}^{x}\frac{-Li_{p}(-t)}{1-t}\mathrm{d}t\\
&=-Li_{1}(x)Li_{p}(-x)+\int_{0}^{x}\frac{Li_{1}(t)Li_{p-1}(-t)}{t}\mathrm{d}t\\
&=-Li_{1}(x)Li_{p}(-x)+Li_{2}(x)Li_{p-1}(-x)-\int_{0}^{x}\frac{Li_{2}(t)Li_{p-2}(-t)}{t}\mathrm{d}t\\
&=\sum_{j=1}^{k}(-1)^{j}Li_{j}(x)Li_{p+1-j}(-x)+(-1)^{k+1}\int_{0}^{x}\frac{Li_{k}(t)Li_{p-k}(-t)}{t}\mathrm{d}t \quad (0 \leq k \leq p)\\
&=\sum_{j=1}^{p}(-1)^{j}Li_{j}(x)Li_{p+1-j}(-x)
+(-1)^{p}\int_{0}^{x}\frac{Li_{p}(t)}{1+t}\mathrm{d}t\,.
\end{align*}
Comparing the coefficients on both sides gives the desired result.
\end{proof}

\begin{Lem}(\cite[p.303-p.304]{Lewin})\label{lem1}
The following formulas are known:
\begin{align*}
&\quad \int_{0}^{t}\frac{\log(a+bt)}{c+et}\mathrm{d}t\\
&=\frac{1}{e}\log\bigg(\frac{ae-bc}{e}\bigg)\log\bigg(\frac{c+et}{c}\bigg)-\frac{1}{e}Li_{2}\bigg(\frac{b(c+et)}{bc-ae}\bigg)
+\frac{1}{e}Li_{2}\bigg(\frac{bc}{bc-ae}\bigg)\,\\
&=\frac{1}{2e}\log^{2}\bigg(\frac{b}{e}(c+et)\bigg)-\frac{1}{2e}\log^{2}\bigg(\frac{bc}{e}\bigg)
+\frac{1}{e}Li_{2}\bigg(\frac{bc-ae}{b(c+et)}\bigg)-\frac{1}{e}Li_{2}\bigg(\frac{bc-ae}{bc}\bigg)\,.
\end{align*}
\end{Lem}

\begin{theorem}\label{thm3}
Let $\lvert x \lvert \leq 1$, then we have
\begin{align*}
\sum_{n=1}^\infty \frac{\overline{H}_{n}}{n}x^{n}
&=-\log(2)\log(1-x)+Li_{2}\bigg(\frac{1}{2}(1-x)\bigg)-Li_{2}\bigg(\frac{1}{2}\bigg)-Li_{2}(-x)\\
&=-\log(1-x)\log(1+x)+\log(2)\log(1+x)-Li_{2}\bigg(\frac{1}{2}(1+x)\bigg)\\
&\quad +Li_{2}\bigg(\frac{1}{2}\bigg)-Li_{2}(-x)\,.
\end{align*}
\end{theorem}
\begin{proof}
Integrating the generating function of $\overline{H}_{n}$, we can write
\begin{align*}
&\quad \sum_{n=1}^\infty \frac{\overline{H}_{n}}{n+1}x^{n+1}=\int_{0}^{x}\frac{\log(1+t)}{1-t}\mathrm{d}t\\
&=-\log(1-x)\log(1+x)+\int_{0}^{x}\frac{\log(1-t)}{1+t}\mathrm{d}t\,.
\end{align*}
Set $a=b=c=1, e=-1$ and $a=c=e=1, b=-1$ in Lemma \ref{lem1} respectively, we have
\begin{align*}
&\int_{0}^{x}\frac{\log(1+t)}{1-t}\mathrm{d}t
=-\log(2)\log(1-x)+Li_{2}\bigg(\frac{1}{2}(1-x)\bigg)-Li_{2}\bigg(\frac{1}{2}\bigg)\,,\\
&\int_{0}^{x}\frac{\log(1-t)}{1+t}\mathrm{d}t=\log(2)\log(1+x)-Li_{2}\bigg(\frac{1}{2}(1+x)\bigg)+Li_{2}\bigg(\frac{1}{2}\bigg)\,.
\end{align*}
Note that
$$
\sum_{n=1}^\infty \frac{\overline{H}_{n}}{n}x^{n}=\sum_{n=1}^\infty \frac{\overline{H}_{n-1}}{n}x^{n}+\sum_{n=1}^\infty \frac{(-1)^{n-1}x^{n}}{n^{2}}\,,
$$
thus we get the desired result.
\end{proof}

\begin{Lem}\label{lem2}
Let $0 \leq x \leq 1$, then we have
\begin{align*}
&\quad \int_{0}^{x}\frac{\log(1-t)\log(1+t)}{t}\mathrm{d}t\\
&=\frac{1}{8}\log(x)\log^{2}(1-\sqrt{x})-\frac{1}{2}\log(x)\log^{2}(1-x)-\frac{1}{2}\log(x)\log^{2}(1+x)\\
&\quad +\frac{1}{3}\log^{3}(1+x)+\frac{1}{2}\log(1-\sqrt{x})Li_{2}(1-\sqrt{x})-\log(1-x)Li_{2}(1-x)\\
&\quad +\log(1+x)Li_{2}\bigg(\frac{1}{1+x}\bigg)-\frac{1}{2}Li_{3}(1-\sqrt{x})+Li_{3}(1-x)-\frac{3}{2}Li_{3}(1)\\
&\quad +Li_{3}\bigg(\frac{1}{1+x}\bigg)\,.
\end{align*}
\end{Lem}
\begin{proof}
We start from
$$
\int_{0}^{x}\frac{\log^{2}(1-t^{2})}{t}\mathrm{d}t\,.
$$
This integral equals to
$$
\int_{0}^{x}\frac{\log^{2}(1-t)}{t}\mathrm{d}t+2\int_{0}^{x}\frac{\log(1-t)\log(1+t)}{t}\mathrm{d}t
+\int_{0}^{x}\frac{\log^{2}(1+t)}{t}\mathrm{d}t\,.
$$
It is obvious that
\begin{align*}
&\quad \int_{0}^{x}\frac{\log^{2}(1-t^{2})}{t}\mathrm{d}t
=\frac{1}{2}\int_{0}^{\sqrt{x}}\frac{\log^{2}(1-t)}{t}\mathrm{d}t\\
&=\frac{1}{4}\log(x)\log^{2}(1-\sqrt{x})+\log(1-\sqrt{x})Li_{2}(1-\sqrt{x})-Li_{3}(1-\sqrt{x})+Li_{3}(1)\,.
\end{align*}
With the help of (\ref{Lewin1}) and (\ref{Lewin2}), we get the desired result.
\end{proof}

\begin{theorem}\label{thm4}
Let $0 \leq x \leq 1$, then we have
\begin{align*}
&\quad \sum_{n=1}^\infty \frac{\overline{H}_{n}^{(2)}}{n}x^{n}\\
&=Li_{2}(-x)\log(1-x)-Li_{3}(-x)+\frac{1}{8}\log(x)\log^{2}(1-\sqrt{x})-\frac{1}{2}\log(x)\log^{2}(1-x)\\
&\quad -\frac{1}{2}\log(x)\log^{2}(1+x)
+\frac{1}{3}\log^{3}(1+x)+\frac{1}{2}\log(1-\sqrt{x})Li_{2}(1-\sqrt{x})\\
&\quad -\log(1-x)Li_{2}(1-x) +\log(1+x)Li_{2}\bigg(\frac{1}{1+x}\bigg)-\frac{1}{2}Li_{3}(1-\sqrt{x})\\
&\quad +Li_{3}(1-x)+Li_{3}\bigg(\frac{1}{1+x}\bigg)-\frac{3}{2}Li_{3}(1)\,,\\
&\quad \sum_{n=1}^\infty \frac{\overline{H}_{n}^{(2)}}{n}(-x)^{n}\\
&=Li_{2}(x)\log(1+x)-Li_{3}(x)+\frac{1}{8}\log(x)\log^{2}(1-\sqrt{x})-\frac{1}{2}\log(x)\log^{2}(1-x)\\
&\quad -\frac{1}{2}\log(x)\log^{2}(1+x)
+\frac{1}{3}\log^{3}(1+x)+\frac{1}{2}\log(1-\sqrt{x})Li_{2}(1-\sqrt{x})\\
&\quad -\log(1-x)Li_{2}(1-x) +\log(1+x)Li_{2}\bigg(\frac{1}{1+x}\bigg)-\frac{1}{2}Li_{3}(1-\sqrt{x})\\
&\quad +Li_{3}(1-x)+Li_{3}\bigg(\frac{1}{1+x}\bigg)-\frac{3}{2}Li_{3}(1)\,.
\end{align*}
\end{theorem}
\begin{proof}
Integrating the generating function of $\overline{H}_{n}^{(2)}$, we can write
\begin{align*}
&\sum_{n=1}^\infty \frac{\overline{H}_{n}^{(2)}}{n+1}x^{n+1}
=\log(1-x)Li_{2}(-x)+\int_{0}^{x}\frac{\log(1-t)\log(1+t)}{t}\mathrm{d}t\,,\\
&\sum_{n=1}^\infty \frac{\overline{H}_{n}^{(2)}}{n+1}(-x)^{n+1}
=\log(1+x)Li_{2}(x)+\int_{0}^{x}\frac{\log(1-t)\log(1+t)}{t}\mathrm{d}t\,.
\end{align*}
Note that
\begin{align*}
\sum_{n=1}^\infty \frac{\overline{H}_{n}^{(2)}}{n}x^{n}=\sum_{n=1}^\infty \frac{\overline{H}_{n-1}^{(2)}}{n}x^{n}+\sum_{n=1}^\infty \frac{(-1)^{n-1}x^{n}}{n^{3}}\,,\\
\sum_{n=1}^\infty \frac{\overline{H}_{n}^{(2)}}{n}(-x)^{n}=\sum_{n=1}^\infty \frac{\overline{H}_{n-1}^{(2)}}{n}(-x)^{n}-\sum_{n=1}^\infty \frac{x^{n}}{n^{3}}\,,
\end{align*}
with the help of Lemma \ref{lem2}, we get the desired result.
\end{proof}

\begin{Example}
Some illustrative examples are as following:
\begin{align*}
&\sum_{n=1}^\infty \frac{\overline{H}_{n}(-1)^{n}}{n}
=-\frac{1}{12}\pi^{2}-\frac{1}{2}\log^{2}(2)\,,\\
&\sum_{n=1}^\infty \frac{\overline{H}_{n}^{(2)}(-1)^{n}}{n}
=-\frac{13}{8}\zeta(3)+\frac{1}{6}\pi^{2}\log(2)\,.
\end{align*}
\end{Example}

\end{document}